\documentclass[smalltitle, tikz]{nmd/article}
\ifnmd
\fi

\title{Asymmetric hyperbolic L-spaces, \\ 
\mbox{}\hspace{1.5cm} Heegaard genus, and Dehn filling}

\author{Nathan M. Dunfield}
\givenname{Nathan}
\surname{Dunfield}
\address{ Dept.~of Math., MC-382 \\
          University of Illinois \\
          1409 W. Green St. \\
          Urbana, IL 61801 \\ 
          USA
}
\email{nathan@dunfield.info}
\urladdr{http://dunfield.info}

\author{Neil R. Hoffman}
\givenname{Neil}
\surname{Hoffman}
\address{School~of Math. and Stat. \\
University of Melbourne\\
Parkville, VIC 3010 \\
Australia}
\email{nhoffman@ms.unimelb.edu.au}
\urladdr{http://ms.unimelb.edu.au/~nhoffman/}

\author{Joan E. Licata}
\givenname{Joan}
\surname{Licata}
\address{Mathematical Sciences Institute\\
John Dedman Bldg 27\\
The Australian National University 0200\\
 Australia}
\email{joan.licata@anu.edu.au}
\urladdr{http://maths-people.anu.edu.au/~licataj/}

\arxivreference{}
\arxivpassword{}

\newcommand{\centercolhead}[1]{\multicolumn{1}{c}{#1}}
\newcommand{\Na}{N_\alpha}
\newcommand{\Nb}{N_\beta}
\newcommand{\HoneZ}[1]{H_1(#1; \Z)}
\newcommand{\HoneZbdryN}{\HoneZ{\partial N}}

\newcommand{\tilt}[1]{\mathrm{Tilt}\!\left(#1\right)}
\newcommand{\interval}[1]{[#1]}
\newcommand{\inta}{\interval{a}}
\newcommand{\intb}{\interval{b}}
\newcommand{\intC}{\interval{C}}
\newcommand{\intCone}{\interval{C_1}}
\newcommand{\intCzero}{\interval{C_0}}
\newcommand{\intCp}{\interval{C^+_1}}
\newcommand{\intCm}{\interval{C^-_1}}

\begin{document}

\begin{abstract} 
  An $L$-space is a rational homology 3-sphere with minimal Heegaard
  Floer homology.  We give the first examples of hyperbolic $L$-spaces
  with no symmetries.  In particular, unlike all previously known
  $L$-spaces, these manifolds are not double branched covers of links
  in $S^3$.  We prove the existence of infinitely many such examples
  (in several distinct families) using a mix of hyperbolic geometry,
  Floer theory, and verified computer calculations.  Of independent
  interest is our technique for using interval arithmetic to certify
  symmetry groups and non-existence of isometries of cusped hyperbolic
  3-manifolds. In the process, we give examples of 1-cusped hyperbolic
  3-manifolds of Heegaard genus 3 with two distinct lens space
  fillings.  These are the first examples where multiple Dehn fillings
  drop the Heegaard genus by more than one, which answers a question
  of Gordon.
\end{abstract}
\maketitle

 \section{Introduction}
\subsection{Asymmetric \emph{L}-spaces}
For a rational homology \3-sphere $M$, the rank of its Heegaard Floer
homology $\HFhat(M)$ is always bounded below by the order of
$\HoneZ{M}$, and $M$ is called an \textit{$L$-space} when this bound is an
equality.  Lens spaces and other spherical manifolds are all
$L$-spaces, but these are by no means the only examples.  In fact, recent
work of Boyer, Gordon, and Watson \cite{BoyerGordonWatson2013} shows
that each of the eight \3-dimensional geometries has an
$L$-space. Their work is part of broader efforts to characterize
$L$-spaces via properties not obviously connected to Heegaard
Floer theory; specifically, they conjecture that a rational homology
sphere is an $L$-space if and only if its fundamental group is not
left-orderable. Although the conjecture has been resolved for seven of
the geometries, it remains open for the important case of hyperbolic
geometry as well as for most manifolds with non-trivial JSJ
decompositions.  

All previous examples of hyperbolic $L$-spaces have come via the
following specific type of surgery construction, and one of our main
results demonstrates that this is a construction of convenience rather
than necessity.  A \emph{strong inversion} of a cusped 3-manifold is
an orientation preserving, order-two symmetry which acts on each cusp
by the elliptic involution; any closed manifold obtained by Dehn
filling inherits this symmetry.  To date, all hyperbolic $L$-spaces
have been constructed by surgery on strongly invertible manifolds, and
moreover, the quotient of the $L$-space by the induced symmetry was always
$S^3$.  Recall that a hyperbolic 3-manifold is \emph{asymmetric} if
its only self-isometry is the identity map; by a deep theorem of
Gabai, this is equivalent to every self\hyp diffeomorphism being
isotopic to the identity \cite{Gabai2001}.  We show the following:
 \begin{theorem}\label{thm:main} 
   There exist infinitely many asymmetric hyperbolic $L$-spaces.  In
   particular, there are hyperbolic $L$-spaces which are neither
   regular covers nor regular branched covers of another \3-manifold.
\end{theorem}
\noindent 
Among $L$-spaces which are \emph{not} double branched covers over
links in $S^3$, hyperbolic examples such as those of
Theorem~\ref{thm:main} are the simplest possible in the sense that any
such $L$-space must have a hyperbolic piece in its prime/JSJ
decomposition.  This is because any graph manifold which is a rational
homology sphere, much less an $L$-space, is a double branched cover
over a link in $S^3$. 
This was proved by Montesinos in
\cite[$\S$7.2]{montesinos1973variedades}; the theorem stated there is
paraphrased 
in the translation below: 
\begin{theorem}[{\cite[$\S$7.2]{montesinos1973variedades}}]\label{thm:Montesinos}
  Let $M$ be a graph manifold whose diagram is a tree with each vertex
  corresponding to a Seifert fibered space over a (punctured) $S^2$ or
  (punctured) $\RP^2$.  Then $M$ is a double branched cover of a link
  $L$ in $S^3.$
\end{theorem}
\noindent
Note the rational homology sphere assumption implies that the diagram
of the graph manifold is a tree. Also, the cases that arise if the
tree is a just single vertex are covered in
\cite[$\S$2-3]{montesinos1973variedades}.

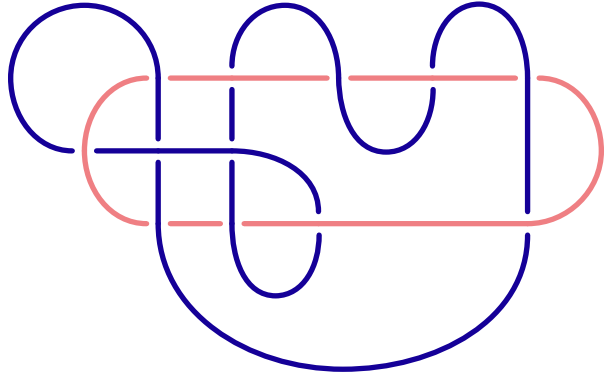
\begin{figure}
  \begin{center}
\definecolor{linkcolor0}{rgb}{0.937, 0.498, 0.514}
\definecolor{linkcolor1}{rgb}{0.082, 0.000, 0.596}
\begin{tikzpicture}[line width=2.0, line cap=round, line join=round, scale=0.8]
  \begin{scope}[color=linkcolor0]
    \draw (8.61, 2.55) .. controls (9.27, 2.55) and (9.82, 3.10) .. 
          (9.82, 3.76) .. controls (9.82, 4.40) and (9.40, 4.97) .. (8.80, 4.97);
    \draw (8.41, 4.97) .. controls (7.96, 4.97) and (7.51, 4.97) .. (7.05, 4.97);
    \draw (7.05, 4.97) .. controls (6.61, 4.97) and (6.16, 4.97) .. (5.71, 4.97);
    \draw (5.32, 4.97) .. controls (4.80, 4.97) and (4.28, 4.97) .. (3.76, 4.97);
    \draw (3.76, 4.97) .. controls (3.43, 4.97) and (3.09, 4.97) .. (2.75, 4.97);
    \draw (2.36, 4.97) .. controls (1.76, 4.97) and (1.34, 4.40) .. (1.34, 3.76);
    \draw (1.34, 3.76) .. controls (1.34, 3.13) and (1.76, 2.55) .. (2.36, 2.55);
    \draw (2.75, 2.55) .. controls (3.02, 2.55) and (3.29, 2.55) .. (3.57, 2.55);
    \draw (3.96, 2.55) .. controls (4.37, 2.55) and (4.78, 2.55) .. (5.18, 2.55);
    \draw (5.18, 2.55) .. controls (6.32, 2.55) and (7.46, 2.55) .. (8.61, 2.55);
  \end{scope}
  \begin{scope}[color=linkcolor1]
    \draw (2.55, 2.55) .. controls (2.55, 1.09) and (4.01, 0.13) .. 
          (5.58, 0.13) .. controls (7.12, 0.13) and (8.61, 0.95) .. (8.61, 2.36);
    \draw (8.61, 2.75) .. controls (8.61, 3.49) and (8.61, 4.23) .. (8.61, 4.97);
    \draw (8.61, 4.97) .. controls (8.61, 5.58) and (8.36, 6.19) .. 
          (7.82, 6.20) .. controls (7.34, 6.20) and (7.05, 5.70) .. (7.05, 5.17);
    \draw (7.06, 4.78) .. controls (7.06, 4.25) and (6.78, 3.74) .. 
          (6.29, 3.74) .. controls (5.76, 3.74) and (5.52, 4.36) .. (5.51, 4.97);
    \draw (5.51, 4.97) .. controls (5.51, 5.59) and (5.19, 6.18) .. 
          (4.63, 6.18) .. controls (4.12, 6.18) and (3.76, 5.71) .. (3.76, 5.17);
    \draw (3.76, 4.78) .. controls (3.76, 4.50) and (3.76, 4.23) .. (3.76, 3.96);
    \draw (3.76, 3.57) .. controls (3.76, 3.23) and (3.76, 2.89) .. (3.76, 2.55);
    \draw (3.76, 2.55) .. controls (3.76, 1.96) and (3.97, 1.35) .. 
          (4.48, 1.35) .. controls (4.94, 1.36) and (5.19, 1.85) .. (5.19, 2.36);
    \draw (5.18, 2.75) .. controls (5.18, 3.40) and (4.48, 3.76) .. (3.76, 3.76);
    \draw (3.76, 3.76) .. controls (3.36, 3.76) and (2.96, 3.76) .. (2.55, 3.76);
    \draw (2.55, 3.76) .. controls (2.22, 3.76) and (1.88, 3.76) .. (1.54, 3.76);
    \draw (1.14, 3.76) .. controls (0.55, 3.76) and (0.13, 4.34) .. 
          (0.13, 4.97) .. controls (0.13, 5.64) and (0.67, 6.18) .. 
          (1.34, 6.18) .. controls (2.01, 6.18) and (2.55, 5.64) .. (2.55, 4.97);
    \draw (2.55, 4.97) .. controls (2.55, 4.64) and (2.55, 4.30) .. (2.55, 3.96);
    \draw (2.55, 3.57) .. controls (2.55, 3.23) and (2.55, 2.89) .. (2.55, 2.55);
  \end{scope}
\end{tikzpicture}
  \end{center}
  \caption{The link used in Theorem~\ref{thm:link} is $L12n1314$ in
    the Hoste-Thistlewaite census. Our framing conventions for Dehn
    filling are
    \protect\raisebox{-0.1cm}{\protect\includegraphics[angle=90,
      scale=0.35]{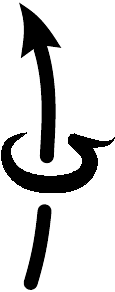}} and are consistent
    with SnapPy \cite{SnapPy}.  Note there is an
    orientation-preserving homeomorphism of $S^3$ which interchanges
    the two components.  }\label{fig:link}
\end{figure}

We prove Theorem~\ref{thm:main} via a combination of hyperbolic
geometry, Heegaard Floer theory, and verified computer calculations.
The proof of Theorem~\ref{thm:main} has two parts, the second of which
is computer-aided.  The first result shows that we need only construct
1-cusped manifolds with certain properties, and the second establishes
the existence of such manifolds.  Here, the \emph{order} of a lens
space is the order of its fundamental group/first homology.

\begin{restatable*}{theorem}{theoremtheory}\label{thm:theory}
  Suppose $M$ is a $1$-cusped hyperbolic \3-manifold.  If $M$ is
  asymmetric and has two lens space Dehn fillings of coprime order,
  then there are infinitely many Dehn fillings of $M$ which are
  asymmetric hyperbolic $L$-spaces. Moreover, $M$ is the complement of
  a knot in an integral homology \3-sphere and fibers over the circle
  with fiber a once-punctured surface.
\end{restatable*}

\begin{restatable*}{theorem}{theoremlink}\label{thm:link}
  There exist infinitely many 1-cusped hyperbolic \3-manifolds which
  are asymmetric and have two lens space fillings of coprime order.
  Specifically, if $N$ is the exterior of the link in
  Figure~\ref{fig:link}, then for all large $k \in \Z$, the $(6k\pm1,
  k)$ Dehn filling on either component of $N$ yields such a manifold.
\end{restatable*}
\noindent
In addition to Theorem~\ref{thm:link}, Theorem~\ref{thm:comp} 
offers a finite number of explicit examples for which the proof is
slightly easier.  A Heegaard diagram of the simplest of these examples
is given in Figure~\ref{fig:v3372}.  

\subsection{Heegaard genus,  Dehn filling, and the Berge conjecture} 
Our second main result answers a question of Gordon
\cite{GordonAIM} regarding the existence of manifolds where multiple
fillings drop the Heegaard genus by more than one:
\begin{corollary}\label{cor:genusdrop}
  There exist infinitely many 1-cusped hyperbolic \3-manifolds of
  Heegaard genus three which admit two distinct lens space fillings.
\end{corollary}
\noindent
This corollary follows immediately from Theorem~\ref{thm:link}, as
manifolds with genus two Heegaard splittings always have symmetries; 
 the examples of Theorem~\ref{thm:link} must have Heegaard
genus exactly three since the link in Figure~\ref{fig:link} is
\3-bridge. 

The interest in $L$-spaces stems in part from open questions about
lens space surgery, with the Berge Conjecture as the chief example.
Another interesting feature of Corollary~\ref{cor:genusdrop} is that
it provides counterexamples to the following generalization of the
Berge Conjecture, since the exterior of any $(1,1)$--knot has Heegaard
genus two:

\begin{conjecture}[{\cite[Conjecture 9]{BakerDoleshalHoffman}}] 
  If knots $K_1 \subset L(p_1, q_1)$ and $K_2 \subset L(p_2,q_2)$ are
  longitudinal surgery duals, then up to reindexing, $K_2$ is a
  $(1,1)$--knot and $p_2 \geq p_1$.
\end{conjecture}
\noindent
We note that these examples do
not contradict the Berge Conjecture itself because they are not
knot complements in $S^3$; see the proof of Theorem~\ref{thm:link} for details.

\subsection{Certifying symmetry groups}

The hard part of proving Theorems~\ref{thm:comp} and \ref{thm:link} is
determining the symmetry groups of 23 cusped hyperbolic \3-manifolds,
in particular, showing that they are asymmetric.  Following Weeks and
collaborators \cite{HenryWeeks1992, Weeks1993convex,
  HodgsonWeeks1994}, we do this by using the Epstein-Penner canonical
cellulation; the symmetry group agrees with the combinatorial
isomorphisms of this cellulation.  For each manifold, we give a
rigorous computer-assisted proof that a certain triangulation is the
canonical cellulation.  We build on the verified computation scheme of
\cite{hikmot2013verified} for proving the existence of hyperbolic
structures.  This scheme replaces floating-point computations subject
to various kinds of errors with interval arithmetic in order to meet
the traditional standards of rigorous proof.  Our method for
certifying a triangulation as canonical is described in detail in
Section~\ref{sect:theory2} and is not specific to the examples here.
In addition, the proofs of Theorems~\ref{thm:comp} and \ref{thm:link}
employ SnapPy \cite{SnapPy} to perform combinatorial computations.
Both SnapPy and the code for \cite{hikmot2013verified} are freely
available; the source code and data files used in the
computer-assisted proofs in this paper are permanently archived at
\cite{ancillary}.

As further context for Theorems~\ref{thm:comp} and \ref{thm:link}, we
note that in general it is quite difficult to show a particular
3-manifold is asymmetric.  Most proofs that \emph{specific} hyperbolic
3-manifolds are asymmetric hinge on computing a hyperbolic invariant
which is not preserved by any possible isometry; see for example the
delicate arguments in \cite{Riley}.  One notable exception is the case
of complements to certain arborescent knots
\cite{bonahonSiebenmann2010}; since knots are determined by their
complements \cite{GordonLuecke1989}, the symmetry group of a knot
complement is the same as that of the pair $(S^3, K)$, where
additional tools apply.  As the referee pointed out to us, the link
$L$ in Theorem~\ref{thm:link} is Montesinos and the symmetry group
of $(S^3, L)$ can be computed by \cite{BoileauZimmermann1987}.  While
this is less information than the symmetry group of the exterior of
$L$, it is possible to leverage this fact to a computer-free proof of
Theorem~\ref{thm:link} and hence Theorem~\ref{thm:main} and
Corollary~\ref{cor:genusdrop}; see Remark~\ref{rem:referee} for
details.  However, this alternative approach does not extend to the
specific examples in Theorem~\ref{thm:comp} of asymmetric 1-cusped
manifolds.

\subsection{Acknowledgements} 
The authors gratefully thank Ken Baker, Francis Bonahon, Craig
Hodgson, Adam Levine, and Jessica Purcell for many helpful
conversations. Dunfield was partially supported by US NSF grant
\#DMS-1106476 and a Simons Fellowship; this work was partially done
while he was visiting the University of Melbourne. Hoffman was
supported by the ARC Grant DP130103694 and thanks the Mathematical
Sciences Institute at the Australian National University for hosting
him during part of this work.  We also thank the referee for their
very helpful comments, especially for pointing us to the work of
\cite{BoileauZimmermann1987} which is discussed in
Remark~\ref{rem:referee}.

\section{Asymmetric L-spaces from cusped manifolds}

This section is devoted to the proof of the following result:
\theoremtheory
\noindent
This theorem follows immediately from the next two lemmas, where in
the second one we set $N \setminus \partial N \cong M$. 

\begin{lemma}\label{lem:hyperdehn}
  Suppose $M$ is an asymmetric 1-cusped hyperbolic \3-manifold.  Then
  all but finitely many Dehn fillings of $M$ are hyperbolic and asymmetric.
\end{lemma}

\begin{lemma}\label{lem:floer}
  Suppose $N$ is a compact \3-manifold with $\partial N$ a torus.  If
  $N$ has two lens space Dehn fillings of 
  coprime order, then $N$ has infinitely many Dehn fillings which are
  $L$-spaces.  Moreover, $N$ is the exterior of a knot in an integral
  homology sphere and fibers over the circle with fiber a surface with
  one boundary component.
\end{lemma}

The proofs of these two lemmas are completely independent and will be
familiar to experts in the areas of \3-dimensional hyperbolic geometry and
Heegaard Floer theory, respectively.

\begin{proof}[Proof of Lemma~\ref{lem:hyperdehn}]
  Our argument here is motivated by \cite{HodgsonWeeks1994}, which
  contains additional details.  The key geometric claim is that, for
  all but finitely many slopes $\alpha$, the Dehn filled manifold
  $M_{\alpha}$ is hyperbolic with the core $c$ of the added solid
  torus being the \emph{unique} shortest closed geodesic in
  $M_{\alpha}$. Since $c$ is the unique geodesic of its length, any
  isometry of $M_{\alpha}$ must send $c$ to itself, setwise if not
  pointwise.  Any isometry of $M_{\alpha}$ thus induces a
  self-diffeomorphism of $M$.  Any symmetry of $M_{\alpha}$ would thus
  give one of the asymmetric manifold $M$, and so $M_{\alpha}$ 
  must also be asymmetric, as desired.

  The geometric claim follows from the proof of the Hyperbolic Dehn Surgery
  Theorem \cite[Theorem 5.8.2]{ThurstonNotes1979} as we now explain.
  Thurston showed that all but finitely many Dehn fillings on $M$
  give closed hyperbolic \3-manifolds whose geometry is very close to
  that of $M$ outside the core curves of the filling solid tori; for further background see \cite[$\S$4.6-4.8]{ThurstonNotes1979}.
  Specifically, for any fixed $\epsilon > 0$, after excluding finitely
  many slopes $\alpha$, we can assume that $M_\alpha$ is hyperbolic
  with the core curve $c$ being a geodesic of length less than
  $\epsilon$ which lives inside a very deep tube whose complement is $(1 +
  \epsilon)$--bi-Lipschitz to a fixed compact subset of $M$. Taking
  $\epsilon$ much smaller than the length of the shortest closed
  geodesic in $M$, it follows that $c$ is the \emph{unique} shortest
  closed geodesic in $M_{\alpha}$. This establishes the geometric
  claim and hence the lemma.
\end{proof}

\begin{proof}[Proof of Lemma~\ref{lem:floer}]  
  We first show that $N$ is the exterior of a knot in an integral homology sphere.  Let
  $\alpha$ and $\beta$ be the given slopes where $\Na$ and $\Nb$ are lens
  spaces.  Since $\HoneZ{\Na} = \HoneZ{N}\big/\langle \alpha \rangle$
  and $\HoneZ{\Nb} = \HoneZ{N}\big/\langle \beta \rangle$ are cyclic
  of coprime order, it follows that $\HoneZ{N}\big/\langle \alpha,
  \beta \rangle$ is trivial and hence that $\HoneZbdryN \to \HoneZ{N}$ is
  surjective; combining this with ``half-lives, half-dies'' for
  $H_1(\partial N; \F_p) \to H_1(N; \F_p)$ for every prime $p$, it
  follows that $\HoneZ{N} \cong \Z$.  Let $\mu \in \HoneZbdryN$ be any
  primitive element whose image generates $\HoneZ{N}$.  Then $N_\mu$
  is an integral homology sphere as desired.
  
  A knot $K$ in a lens space $L$ is \textit{primitive} if $[K]$ generates
  $\HoneZ{L}$.  Since $\HoneZbdryN$ surjects onto $\HoneZ{N}$, it follows
  that $N$ is the exterior of primitive knots in $\Na$ and $\Nb$; Theorem
  6.5 of \cite{BBCW2012} then implies that $N$ fibers over the
  circle. An easy consequence of the surjectivity of $\HoneZbdryN \to
  \HoneZ{N}$ is that the fiber has only one boundary component.

  It remains to show that $N$ has infinitely many $L$-space fillings,
  which is a standard consequence of the exact triangle in
  Heegaard Floer homology, specifically:

  \begin{proposition}[{\cite[Prop 2.1]{OzsvathSzabo2005}}]
  \label{prop:osz} 
    Suppose $\{\eta, \nu\}$ are a basis for $\HoneZbdryN$ and $N_\eta$,
    $N_\nu$, and $N_{\eta + \nu}$ are all rational homology spheres with
    \begin{equation}\label{eq:homcond}
    \abs{\HoneZ{N_{\eta+\nu}}}= \abs{\HoneZ{N_\eta}}+\abs{\HoneZ{N_{\nu}}\vphantom{\big|}}.
    \end{equation}
    If $N_\eta$ and $N_\nu$ are $L$-spaces, so is $N_{\eta + \mu}$.
\end{proposition}

\noindent
As elements of $\HoneZbdryN$, orient $\alpha$ and $\beta$ so that the
cone $C = \setdef{ a \alpha + b \beta }{a, b \in \Z_{>0}}$ is disjoint
from the kernel of $\HoneZbdryN \to \HoneZ{N}$.  It is enough to show
that every primitive lattice point in $C$ corresponds to an $L$-space
filling.  Notice first that the homological picture of $(N, \partial
N)$ developed above means that the map $C \to \N$ which sends $\gamma \to
\abs{\HoneZ{N_\gamma}}$ is the restriction of a linear function.  In
particular, condition (\ref{eq:homcond}) will always hold on $C$.  By
the Cyclic Surgery Theorem \cite{CGLS}, the geometric intersection
number $\alpha \cdot \beta$ is $1$, and hence we may apply
Proposition~\ref{prop:osz} to see that $N_{\alpha + \beta}$ is an
$L$-space.  Repeating this argument inductively with the basis
$\pair{n \alpha + \beta, \ \alpha}$ yields an infinite collection of
$L$-space fillings on $N$, proving the lemma.  One can extend this to
all primitive vectors in $C$ with a little more thought, and a complete
answer to which $N_\eta$ are $L$-spaces is given in
\cite{Rasmussen2007}.
\end{proof}

\section{Certifying canonical triangulations}\label{sect:theory2}

Triangulations are a basic tool in \3-manifold topology, especially 
its algorithmic and computational aspects, and the use of ideal
triangulations to study hyperbolic structures on \3-manifolds goes
back to Thurston \cite{ThurstonNotes1979}.  Although every manifold
has infinitely many triangulations, a cusped hyperbolic \3-manifold $M$
has a unique \emph{canonical ideal cellulation} which is defined
solely in terms of its geometry.  Generically ---including in all the
examples here--- this cellulation is an ideal triangulation, called the
\emph{canonical triangulation}. 

Introduced by Epstein and Penner \cite{EpsteinPenner1988}, the
canonical cellulation is defined by first embedding the universal
cover $\H^3$ of $M$ into $(3+1)$-dimensional Minkowski space.  Choose
disjoint horotorus neighborhoods of each cusp in $M$ which all have
the same volume.  Upstairs in $\H^3$, these neighborhoods lift
to a $\pi_1(M)$--invariant packing of horoballs.  In the Minkowski
model, each horoball $B$ has a corresponding lightcone vector $v_B$,
where $B = \setdef{ w \in \H^3 }{v_B \cdot w \leq -1}$.  The convex
hull of the lightcone vectors associated to the set of cusps has a
natural cellulation of its boundary, and  projecting this radially defines a cellulation
of $\H^3$.  Since this cellulation is preserved both by the action of $\pi_1(M)$
and also by the lifts of isometries of $M$, it descends to a
cellulation of $M$ which is preserved by its isometry group; in
particular, we get the following key tool:

\begin{corollary}[{\cite{HenryWeeks1992}}]\label{cor:asym}
  The elements of the isometry group of $M$ correspond precisely to
  the combinatorial isomorphisms of its canonical cellulation.  In
  particular, if the canonical cellulation has no nontrivial
  combinatorial isomorphisms, then $M$ is asymmetric.
\end{corollary}
 
From now on, let $\cT$ denote an ideal triangulation of $M$ where each
topological tetrahedron has been assigned a shape: an isometry type of
an ideal tetrahedron with geodesic sides in $\H^3$.  Each shape is
specified by a complex number, and these numbers must satisfy certain
polynomial conditions which ensure that these geometric tetrahedra
glue up to give the complete hyperbolic structure on $M$
\cite{ThurstonNotes1979, Weeks2005}.

In \cite{Weeks1993convex}, Weeks gave an easy way to check whether
such a given geometric ideal triangulation is canonical.  Let $X$ be
one of the ideal tetrahedra, and label its vertices $\{0, 1, 2, 3\}$.
For some fixed horotorus cross section of the cusp near vertex $i$,
let $R_i^X$ denote the circumradius of the cross section and let
$\theta^X_{ij}$ denote the dihedral angle of the edge from vertex $i$
to vertex $j$.  For the face $F$ of $X$ opposite vertex $i$, define
\begin{equation}\label{eq:tiltdef}
\tilt{X, F} =  R^X_i - \sum_{k\neq i} R^X_k \cos \theta^X_{ik} 
\end{equation}
If $Y$ is the other tetrahedron sharing $F$ as a face, set $\tilt{F}
\assign \tilt{X, F} + \tilt{Y, F}$. 
With this notation, Weeks' criterion is as follows:
\begin{theorem}[{\cite[Prop 3.1, Thm 5.1]{Weeks1993convex}}]
\label{thm:weeks} 
A geometric ideal triangulation $\cT$ of a cusped manifold $M$ is its
canonical cellulation if and only if every face $F$ of $\cT$ has
$\tilt{F} < 0$.
\end{theorem}
\noindent
Geometrically, the gluing at the face $F$ is convex, flat, or concave,
depending on whether $\tilt{F}$ is negative,
zero, or positive. 

\subsection{Finding the canonical triangulation}  
We next explain how SnapPy attempts to find the canonical cellulation.
This gives context for our results and highlights the necessity of
a verified computation by showing what could go wrong. However, the
reader interested only in the proofs of our results can safely
skip ahead to \S \ref{sec:certhyp}.  
   
In \cite{Weeks1993convex}, Weeks gave a procedure, implemented in
\cite{SnapPy},  to transform an arbitrary geometric
triangulation $\cT$ of a hyperbolic \3-manifold $M$ into the canonical
cellulation.  Neglecting for the moment the
numerical issues inherent in floating-point arithmetic, his procedure is the following:
\begin{enumerate}
\renewcommand{\labelenumi}{(\arabic{enumi}) }
\renewcommand{\theenumi}{\arabic{enumi}}
\item \label{alog:test} If $\tilt{F} < 0$ for every face of $\cT$, then $\cT$ itself is
  the canonical triangulation by Theorem~\ref{thm:weeks}.  If
  $\tilt{F} \leq 0$ for every face, then $\cT$ is a tetrahedral
  subdivision of the canonical cellulation.  In either case, the
  procedure terminates.  

\item \label{alog:2to3} If there is a face $F$ with $\tilt{F} > 0$ and with the property that performing a
2-to-3 Pachner move on $F$ creates only positively 
  oriented tetrahedra, then replace $\cT$ with the result of this
  2-to-3 move and go back to Step~\ref{alog:test}.

\item \label{alog:3to2} If there is a valence three edge $E$ of $\cT$ with a face $F$
  incident to $E$ having $\tilt{F} \geq 0$, then replace $\cT$ with the
  result of the 3-to-2 Pachner move on $E$ and go back to
  Step~\ref{alog:test}.

\item If some face $F$ has $\tilt{F} > 0$ but no moves permitted in
  Steps~\ref{alog:2to3} and \ref{alog:3to2} are possible, do a
  sequence of random Pachner moves to replace $\cT$ with a different
  geometric triangulation and return to Step~\ref{alog:test}.
\end{enumerate}

While in practice this procedure almost always succeeds in finding the
canonical cellulation, it is not known to terminate with probability
1.  More significantly for us, even when it does terminate,
floating-point issues may cause the cellulation returned \emph{not} to
be canonical.  Specifically, as the shapes are known only
approximately and round-off errors may accumulate, SnapPy may conclude
erroneously that $\tilt{F} \leq 0$ for all $F$.  This is not merely a
theoretical concern.  For example, there is a certain 16 tetrahedra
triangulation of the exterior to the link $L10a154$ (included in
\cite{ancillary}) where SnapPy identifies the wrong cellulation as
canonical; in this case, the actual canonical cellulation has
non-tetrahedral cells, which is the hardest case because some tilts
are zero.

\subsection{Certifying hyperbolic structures}\label{sec:certhyp}

Before rigorously finding the canonical triangulation, we must first
certify the existence of a hyperbolic structure.  For this we used the
verification scheme of \cite{hikmot2013verified} which replaces
floating-point computations with rigorous interval arithmetic. In
interval arithmetic, a number $z\in\C$ is partially specified by
giving a rectangle with vertices in $\Q(i)$ which contains $z$.
Because the vertices are rational, such intervals can be exactly
stored on a computer and rigorously combined by the operations
$+,-,\cdot,/$ to create other such intervals.  The cost is that the
sizes of the rectangles grow with the number of operations.  Given an
ideal triangulation $\cT$ of a \3-manifold $M$, in favorable
circumstances the verification scheme of \cite{hikmot2013verified}
produces an interval for each tetrahedral shape, together with a proof
that the actual hyperbolic structure has shapes lying in those
intervals.

\subsection{Certifying canonical triangulations}\label{sec:certone}

We now explain how to extend the work of \cite{hikmot2013verified} to
rigorously certify a triangulation $\cT$ of $M$ as canonical.  The
basic idea is to use interval arithmetic when checking the hypotheses
of Theorem~\ref{thm:weeks}, starting from the guaranteed shape
intervals produced by \cite{hikmot2013verified}.  Note that for a real
interval $r$, it makes sense to say that say $r < 0$ when both of the
endpoints of $r$ are negative.  (In contrast, there is no notion of
equality for intervals since an interval is just a stand-in for some
unknown number inside it.)  Thus if we compute $\tilt{F}$ as a real
interval from the guaranteed shape intervals, we can potentially
certify that $\tilt{F} < 0$ as required by Theorem~\ref{thm:weeks}.
From (\ref{eq:tiltdef}), one sees that it suffices to compute the
quantities $R^X_i$ and $\cos\big(\theta^X_{i,j}\big)$.

We begin with the easier case where $M$ has a single cusp. We
construct a particular cusp cross section by first choosing a corner
of a fixed tetrahedron and then selecting a horospherical Euclidean
triangle whose first side has length 1.  The (known) shape of the
tetrahedron determines the other two sides of this cusp triangle, and
from there, one can propagate the cusp cross section to adjacent
tetrahedra.  Since there is only one cusp, this initial choice
determines the whole cross section.  The resulting ``cusp cross
section'' could be too large to be embedded, but it represents an
actual cross section up to a uniform dilation; since
(\ref{eq:tiltdef}) is homogenous in the $R^X_i$, this has no effect on
checking the hypotheses of Theorem~\ref{thm:weeks}.

The quantity $R^X_i$ is the circumradius of the corresponding cusp triangle,
and the circumradius of a triangle may be computed from the lengths of
its edges using only the operations $+,-,\cdot,/$ and
$\sqrt{\vphantom{\alpha} \quad }$, all of which are supported by the
interval arithmetic scheme of \cite{hikmot2013verified}; see $\S$3.1
of that paper for details.  The cosines of the dihedral angles that
appear in (\ref{eq:tiltdef}) can be similarly computed; if the $(i,
j)$ edge has shape $z=a+bi$, we have $\cos\big(\theta^X_{i,j}\big) =
\cos\big(\arg (z)\big) = a\big/\sqrt{a^2+b^2}$.

Thus in the 1-cusped case, we can compute tilt intervals from the
initial shape intervals and hence potentially apply
Theorem~\ref{thm:weeks} in a rigorous way; we will do precisely this
for the 22 manifolds of Theorem~\ref{thm:comp}.  

\subsection{Multiple cusps} When $M$ has multiple cusps, as in the proof of Theorem~\ref{thm:link}, there is an additional
subtlety.  Specifically, the canonical cellulation is defined in terms
of cusp cross sections which \emph{all have the same area}.  As mentioned above,
interval arithmetic does not support the notion of equality.  In order to describe our solution to this issue,  we must
first introduce some more precise notation.  Letting
$\inta$ and $\intb$  denote intervals, we say that $\inta < \intb$
if $a < b$ for all $a \in \inta$ and $b\in \intb$.  We extend this
notation, using $\intC$ to denote a cusp cross section which is
computed by interval arithmetic from the guaranteed shape data as in
\S \ref{sec:certone}, and
we say that an actual cusp cross section $C$ lies in $\intC$ if all its
Euclidean triangles have side lengths in the corresponding interval
side lengths of $\intC$.  In particular, if $C$ is in $\intC$, then
$\Area(C)$ is in $\Area(\intC)$, where the latter is  an honest interval.

For notational simplicity, let us start with the case where $M$ has
two cusps.  Let $\intCzero$ and $\intCone$ be cusp cross sections
constructed from the shape data as above.  Scale $\intCone$ to create
$\intCm$ and $\intCp$ where the following holds in the interval sense:
\begin{equation}\label{eq:areas}
\Area\big(\intCm\big) < \Area\big(\intCzero\big) < \Area\big(\intCp\big)  
\end{equation}
If $F$ is a face of $\cT$, we use $\tilt{F, \intCzero, \intCone}$ to
denote the tilt interval of $F$ with respect to the cusps
$\intCzero$ and $\intCone$ computed as in \S \ref{sec:certone}.
The proof of Theorem~\ref{thm:link} rests on the following:
\begin{proposition}\label{prop:multicanon}
  Suppose $\cT$ is a geometric triangulation of a \2-cusped manifold
  $M$ with guaranteed shape intervals, and suppose further  that $\intCzero, \intCm,$
  and $\intCp$ are cusp cross sections satisfying (\ref{eq:areas}).
  If for every face $F$ of $\cT$ we have $\tilt{F, \intCzero, \intCm}
  < 0$ and $\tilt{F, \intCzero, \intCp} < 0$, then $\cT$ is the
  canonical triangulation of $M$.
\end{proposition}

\begin{proof}
  Fix actual cusp cross sections $C_0$, $C_1^+$, $C_1^-$ in
  $\intCzero, \intCm, \intCp$, respectively.  The hypotheses imply
  $\Area(C_1^-) < \Area(C_0) < \Area(C_1^+)$. Let $C_1'$ be an actual
  cross section for the second cusp with $\Area(C_1') = \Area(C_0 )$.
  Thinking of cusp cross sections as vectors
  whose coordinates are the circumradii of their constituent Euclidean
  triangles, we can view $C_1'$ as a convex combination
  \[
  C_1' = (1-t) \cdot C_1^- + t \cdot C_1^+ \mtext{for some $t \in
    (0,1)$.}
  \]
  For any face $F$, the function $\tilt{F, C_0, \ \cdot \ }$ is linear
  in the remaining input, and so since both $\tilt{F, C_0, C_1^-}$ and
  $\tilt{F, C_0, C_1^+}$ are negative by our hypotheses, we must have
  $\tilt{F, C_0, C_1'} < 0$.  In particular, since $\Area(C_0) =
  \Area(C_1')$, Theorem~\ref{thm:weeks} now implies that $\cT$ is
  canonical.
\end{proof}
Proposition~\ref{prop:multicanon}  extends easily to manifolds with
three or more cusps.  First fix some $\intCzero$ for the first cusp,
and then for each further cusp, choose a pair of cross sections
$\interval{C_n^\pm}$ with 
\[
\Area\big(\interval{C_n^-}\big) < \Area\big(\intCzero\big) <\Area\big(\interval{C_n^+}\big).
\]
If for every face $F$ and every pattern of signs $\{ \epsilon_n\}$ one has 
\[
\tilt{\interval{C_0}, \interval{C_1^{\epsilon_1}}, \interval{C_2^{\epsilon_2}}, \ldots,
  \interval{C_m^{\epsilon_m}}} < 0, 
\]
then $\cT$ must be canonical.  The point is again that the
actual equal area cross sections are convex combinations of cross
sections in the $\interval{C_n^\pm}$, all of which have negative tilts
when combined with $[C_0]$.

\begin{remark}
  The subsequent paper
  \cite{FominykhGaroufalidisGoernerTarkaevVesnin2015} gives an elegant
  simplification of our technique in the multicusped case (see their
  Section 3.4) and they provide an implementation of their approach
  which works for any number of cusps.
\end{remark}

\subsection{Canonical cellulations with more complicated cells}

Canonical cellulations are generically triangulations, but it would
be useful to be able to certify canonicity of cellulations with more
complicated cells, especially as these include some of the most
symmetric examples.  It is unclear whether this can be done
directly in the context of interval arithmetic, since  the lack of
equality testing means we can not be sure that some tilt is precisely
zero, exactly the condition that leads to non-tetrahedral cells.  In small cases,
one should be able to use exact arithmetic in a number field to deal
with this, as in \cite{CoulsonGoodmanHodgsonNeumann2000}, but the
interval arithmetic techniques of \cite{hikmot2013verified} can be
successfully applied to \emph{much} more complicated manifolds.

\section{Asymmetric manifolds with lens space fillings}\label{sect:comp}

Before proving Theorem~\ref{thm:link}, we warm up with the following
easier and more concrete result, which, when combined with
Theorem~\ref{thm:theory}, also suffices to prove
Theorem~\ref{thm:main}.
\begin{restatable}{theorem}{theoremcomp}\label{thm:comp}
  Table~\ref{table:examples} lists 22 distinct 1-cusped hyperbolic
  \3-manifolds which are asymmetric and have two lens space fillings
  of coprime order.
\end{restatable}
\noindent
We provide a rigorous computer-assisted proof of
Theorem~\ref{thm:comp} using SnapPy \cite{SnapPy}, the verification scheme
of \cite{hikmot2013verified}, and the techniques given in 
Section~\ref{sect:theory2}.  These examples were found in the census
of 1-cusped hyperbolic \3-manifolds with at most 9 tetrahedra
\cite{Burton2014, CallahanHildebrandWeeks1999} by a 
brute-force search through these 59{,}107 manifolds.
\begin{table}
\small
  \begin{center}
    \begin{tabular}{lrllcccc}
\toprule
    \centercolhead{$M$} &  \centercolhead{\#tets} &  \centercolhead{$M_{(1,0)}$} &    \centercolhead{$M_{(0,1)}$} &  $g$ & $\mathrm{vol}(M)$ & systole \\

\midrule
      $v3372^*$ &      7 &    $L(7, 1)$ &   $L(19, 7)$ &  10 &  6.541194 &  0.952884 \\
     $t10397$ &      8 &   $L(11, 2)$ &   $L(14, 3)$ &  12 &  6.880362 &  0.911798 \\
     $t10448$ &      8 &   $L(17, 5)$ &   $L(29, 8)$ &  15 &  6.891314 &  0.716411 \\
     $t11289^*$ &      8 &   $L(11, 2)$ &   $L(26, 7)$ &  15 &  7.084874 &  0.576033 \\
    \hline 
    $t11581$ &      8 &    $L(7, 1)$ &  $L(31, 12)$ &  16 &  7.180413 &  0.767839 \\
     $t11780$ &      8 &   $L(23, 7)$ &    $L(6, 1)$ &  12 &  7.232671 &  0.643558 \\
     $t11824$ &      8 &  $L(34, 13)$ &   $L(19, 4)$ &  19 &  7.246332 &  0.480409 \\
     $t12685$ &      8 &   $L(14, 3)$ &   $L(29, 8)$ &  18 &  7.674889 &  0.693829 \\
     \hline

$o9_{34328}^*$ &     10 &   $L(13, 2)$ &  $L(34, 13)$ &  19 &  7.529794 &  0.312418 \\
 $o9_{35609}$ &     10 &  $L(50, 19)$ &   $L(29, 8)$ &  27 &  7.631975 &  0.237482 \\
 $o9_{35746}^*$ &     10 &   $L(17, 3)$ &  $L(41, 12)$ &  24 &  7.642118 &  0.238001 \\
 $o9_{36591}$ &      9 &  $L(55, 21)$ &   $L(31, 7)$ &  31 &  7.707673
 &  0.188586 \\
 \hline

 $o9_{37290}$ &      9 &  $L(31, 12)$ &   $L(19, 4)$ &  22 &  7.762770 &  0.442218 \\
 $o9_{37552}$ &      9 &   $L(35, 8)$ &   $L(13, 3)$ &  18 &  7.781895 &  0.408545 \\
 $o9_{38147}$ &      9 &  $L(29, 12)$ &  $L(41, 11)$ &  27 &  7.831770 &  0.392648 \\
 $o9_{38375}$ &      9 &   $L(17, 3)$ &   $L(29, 8)$ &  24 &  7.851404
 &  0.349858 \\
 \hline
 $o9_{38845}$ &      9 &   $L(13, 2)$ &   $L(18, 5)$ &  15 &  7.896384 &  0.770335 \\
 $o9_{39220}$ &     10 &   $L(13, 2)$ &  $L(46, 17)$ &  28 &  7.930877 &  0.304931 \\
 $o9_{41039}$ &     10 &   $L(13, 2)$ &   $L(21, 8)$ &  16 &  8.122543 &  0.916284 \\
 $o9_{41063}$ &      9 &   $L(26, 7)$ &  $L(41, 11)$ &  30 &  8.126169
 &  0.386869 \\
\hline 

 $o9_{41329}$ &      9 &   $L(34, 9)$ &  $L(49, 18)$ &  34 &  8.159350 &  0.364220 \\
 $o9_{43248}$ &     10 &   $L(37, 8)$ &   $L(18, 5)$ &  23 &  8.444914 &  0.689245 \\
\bottomrule
\end{tabular}

  \end{center}
  \caption{The 22 manifolds of Theorem~\ref{thm:comp}.  Here, ``\#tets'' refers to the
    canonical triangulation supplied in \cite{ancillary} and $g$ is the genus of the fibration
    of $M$ over the circle (whose existence follows from
    Theorem~\ref{thm:theory}) computed via the Alexander
    polynomial.  The lens spaces were identified using Regina
    \cite{Regina}.  The manifolds marked with a $*$ also appear in
    Theorem~\ref{thm:link}.    The
    data is all rigorous with the exception of the volume and systole
    columns, which were approximated numerically, as the methods of
    \cite{hikmot2013verified} have not yet been extended to those
    quantities.  Note that
    none of these manifolds are knot complements in $S^3$, since the
    pair of lens space surgeries have fundamental groups whose orders
    differ by more than one.}\label{table:examples}
\end{table} 
\begin{figure}
  \vspace{-0.7cm}
  \begin{center}
    \definecolor{lightblue}{rgb}{0.7,0.83,0.97}
\definecolor{mediumblue}{rgb}{0.62,0.75,0.88}

\usetikzlibrary{decorations.markings}

\tikzset{alpha/.style={color=lightblue, line width=2pt,
  dash pattern=on 5pt off 5pt on 1pt off 5pt,
  line cap=round,
  decoration={
  markings,
  mark=at position #1 with {\arrow{>}}},postaction={decorate}}}

\tikzset{beta/.style={color=nmdlight, line width=2.0pt, 
  dash pattern=on 3pt off 5pt,
  line cap=round,
  decoration={
  markings,
  mark=at position #1 with {\arrow{>}}},postaction={decorate}}}

\tikzset{Rone/.style={color=black, line width=2.0pt,  
  decoration={
    markings,
    mark=at position #1 with {\arrow{>}}},postaction={decorate}}}

\tikzset{Rtwo/.style={color=nmdmedium, line width=2.0pt,  
  decoration={
    markings,
    mark=at position #1 with {\arrow{>>}}},postaction={decorate}}}


\begin{tikzpicture}[scale=1.05, x=0.14cm, y=0.14cm]
    \node at (38.8,28.0) {$R_1$};
    \node at (44.3,51.9) {$R_2$};
    \node at (81.9,38.0) {$\beta$};
    \node at (88.0,49.1) {$\alpha$};

    \draw[Rtwo=0.7] (13.9,57.4) 
      .. controls (0.2,49.9) and (2.0,26.4)
      .. (8.4,22.2);
    \draw[alpha=0.54] (13.8,53.9)
      .. controls (7.6,47.1) and (4.9,34.5)
      ..  (10.2,25.0);
    \draw[Rone=0.7]  (13.0,26.6)        
        .. controls (10.4,35.7) and (11.9,44.8)
        .. (16.3,52.1); 
    \draw[color=nmdmedium, line width=2pt] (18.8,51.0) 
      .. controls (18.8,51.0) and (14.4,40.7)
      .. (16.5,28.0);
    \draw[Rtwo=1.0, line cap=round]   (15.95,34.5) -- (15.95,34.0) ;

    \draw[Rtwo=0.75] (37.0,23.3)
      .. controls (35.0,25.7) and (26.6,28.1)
      .. (20.4,25.3);
   \draw[beta=0.64] (35.5,20.8)
       .. controls (31.8,22.5) and (26.0,23.2)
       .. (21.9,22.1);
   \draw[Rtwo=0.70]  (35.2,18.3) 
      .. controls (31.4,17.6) and (26.2,17.5)
      .. (20.8,19.5);
   \draw[Rone=0.6] (37.1,16.3)        
      .. controls (33.3,13.9) and (24.0,14.1) 
      .. (18.3,17.4); 

      \draw[Rone=0.79] (38.3,45.2) 
         .. controls (36.5,48.3) and (31.4,53.3)
         .. (27.0,55.2);
    \draw[alpha=0.65] (25.6,52.7)
      .. controls (30.7,48.2)
      .. (35.9,43.0);
    \draw[Rtwo=0.5] (23.5,51.2)
      .. controls (27.6,45.7)
      .. (34.6,40.0);
    \draw[beta=0.675] (20.9,50.5)
       .. controls (22.2,46.4) and (26.8,40.9) 
       .. (34.0,37.4);

    \draw[Rone=0.7]  (24.5,61.9)
         .. controls (34.3,65.0) and (45.2,62.5)
         .. (51.4,59.7);
    \draw[Rtwo=0.70]  (49.8,56.1)
       .. controls (43.4,58.4) and (34.5,60.1)
       .. (27.5,58.6);

    \draw[Rone=0.7] (41.3,25.1) --  (41.1,31.7);

    \draw[Rtwo=0.65] (50.8,53.6) .. controls (46.3,49.8) 
      .. (43.2,45.6);
    \draw[beta=0.77] (51.4,50.5) --
         (45.3,43.8);
    \draw[Rtwo=0.6] (53.9,49.5) .. controls (51.1,45.7)
      .. (46.7,41.4);
   
    \draw[alpha=0.620] (63.9,25.3)
      .. controls (60.7,31.5) and (54.9,36.4) 
      .. (48.1,38.4);
    \draw[Rone=0.7] (47.0,35.6)        
        .. controls (53.2,33.2) and (58.5,28.8)
        .. (61.1,24.7); 
    \draw[Rtwo=0.40] (44.7,33.4) .. controls (50.0,31.5) and (58.0,24.5) 
      .. (59.0,22.8);

    \draw[alpha=0.74, dash phase=1.5pt] (58.2,20.2)
      .. controls (55.0,22.2) and (51.5,23.0)
      .. (46.9,21.9); 
    \draw[Rone=0.7]  (46.9,18.0)
      .. controls (49.8,16.9) and (54.2,15.9) 
      .. (58.8,17.2);

    \draw[Rtwo=0.40] (45.0,15.8) 
       .. controls (49.9,12.0) and (57.3,11.9)
       .. (61.7,15.2);

   \draw[Rone=0.6] (58.8,50.4)
      .. controls (64.5,47.2) and (69.5,32.5)
      .. (66.6,24.5);
   \draw[Rtwo=0.40] (62.0,53.5) 
      .. controls (70.5,47.0) and (73.9,32.5)
      .. (69.4,22.5); 

    \draw[alpha=0.3, dash phase=-3pt] (59.4,60.3) 
     .. controls (64.1,65.9) and (68.1,70.9)
     .. (76.7,68.8);
     \draw[alpha=0.54, dash phase=3pt] (76.7,68.8)
     .. controls (82.7,68.0) and (93.7,53.6) 
     .. (90.7,30.0);
    \draw[alpha=0.513,  dash phase=12pt] (90.7,30.0)
     .. controls (86.0,-6.5) and (44.0,-4.4)
     .. (40.6,14.5);

    \draw[beta=0.7, dash phase=-4.5pt] (21.0,63.9)
       .. controls (32.2,78.5) and (55.6,76.8)
       .. (68.3,66.0);
    \draw[beta=0.55, dash phase=-0.5pt] (68.3,66.0)
       .. controls (83.0,53.0) and (85.1,23.9)
       .. (69.8,18.7); 

  \begin{scope}[draw=black, fill=white, line width=2.0pt, font=\normalsize]
    \filldraw (20.5,57.3) circle [radius=7.1] node {$A$};
    \filldraw (40.9,38.7) circle [radius=7.1] node {$a$};
    \filldraw  (15.1,21.7) circle [radius=6.7] node {$b$};
    \filldraw  (56.2,55.4) circle [radius=6.7] node {$B$};
    \filldraw (64.1,19.6) circle [radius=6.2] node {$c$};
    \filldraw  (41.3,19.4) circle [radius=6.2] node {$C$};
  \end{scope}

  \begin{scope}[font=\footnotesize]
    \node at (41.1,33.3) {$1$};
    \node at (44.3,34.4) {$2$};
    \node at (46.0,36.2) {$3$};
    \node at (46.7,38.5) {$4$};
    \node at (46.2,41.0) {$5$};
    \node at (44.5,42.9) {$6$};
    \node at (42.2,43.9) {$7$};
    \node at (38.9,43.8) {$8$};
    \node at (36.7,42.3) {$9$};
    \node at (35.8,40.0) {$10$};
    \node at (35.7,37.3) {$11$};

    \node at (24.8,61.2) {$1$};
    \node at (26.1,58.6) {$2$};
    \node at (26.1,56.1) {$3$};
    \node at (25.0,53.9) {$4$};
    \node at (23.25,52.4) {$5$};
    \node at (21.1,52.0) {$6$};
    \node at (19.0,52.2) {$7$};
    \node at (17.1,53.1) {$8$};
    \node at (15.5,54.8) {$9$};
    \node at (15.2,57.5) {$10$};
    \node at (21.0,62.7) {$11$};

    \node at (18.5,17.6) {$1$};
    \node at (20.1,19.7) {$2$};
    \node at (20.6,22.1) {$3$};
    \node at (19.5,24.9) {$4$};
    \node at (16.5,26.7) {$5$};
    \node at (12.8,26.3) {$6$};
    \node at (10.7,24.2) {$7$};
    \node at (9.7,21.7) {$8$};

    \node at (59.2,51.2) {$1$};
    \node at (54.5,50.6) {$2$};
    \node at (52.5,51.9) {$3$};
    \node at (51.3,53.6) {$4$};
    \node at (50.9,55.9) {$5$};
    \node at (52.2,58.8) {$6$};
    \node at (59.3,59.5) {$7$};
    \node at (61.2,53.7) {$8$};

    \node at (66.2,23.8) {$1$};
    \node at (64.0,24.2) {$2$};
    \node at (61.8,23.8) {$3$};
    \node at (60.0,22.2) {$4$};
    \node at (59.3,19.9) {$5$};
    \node at (59.8,17.2) {$6$};
    \node at (61.7,15.5) {$7$};
    \node at (69.1,19.0) {$8$};
    \node at (68.5,21.8) {$9$};

    \node at (41.4,23.9) {$1$};
    \node at (45.6,21.9) {$2$};
    \node at (46.1,18.4) {$3$};
    \node at (44.5,16.1) {$4$};
    \node at (40.9,14.8) {$5$};
    \node at (37.6,16.3) {$6$};
    \node at (36.6,18.3) {$7$};
    \node at (36.7,20.8) {$8$};
    \node at (38.0,22.7) {$9$};

  \end{scope}

\end{tikzpicture}
  \end{center}
  \vspace{-1.75cm}
  \caption{A Heegaard diagram for the first manifold $v3372$ in
    Table~\ref{table:examples}, corresponding to 
    $\big\langle a,b,c \  \big| \ R_1 := ab^{-1}a^{-2}c^2bc = 1, \  R_2 := aba^{-1}c^2ba^2bcb = 1 \big\rangle$. 
    Also shown are
    the slopes $\alpha = c^{-2}a^2b$ and $\beta = cba^2$
    which give lens spaces $L(7,1)$ and $L(19,7)$, 
    oriented so any positive combination of them gives an
    $L$-space.
  }\label{fig:v3372}
\end{figure}
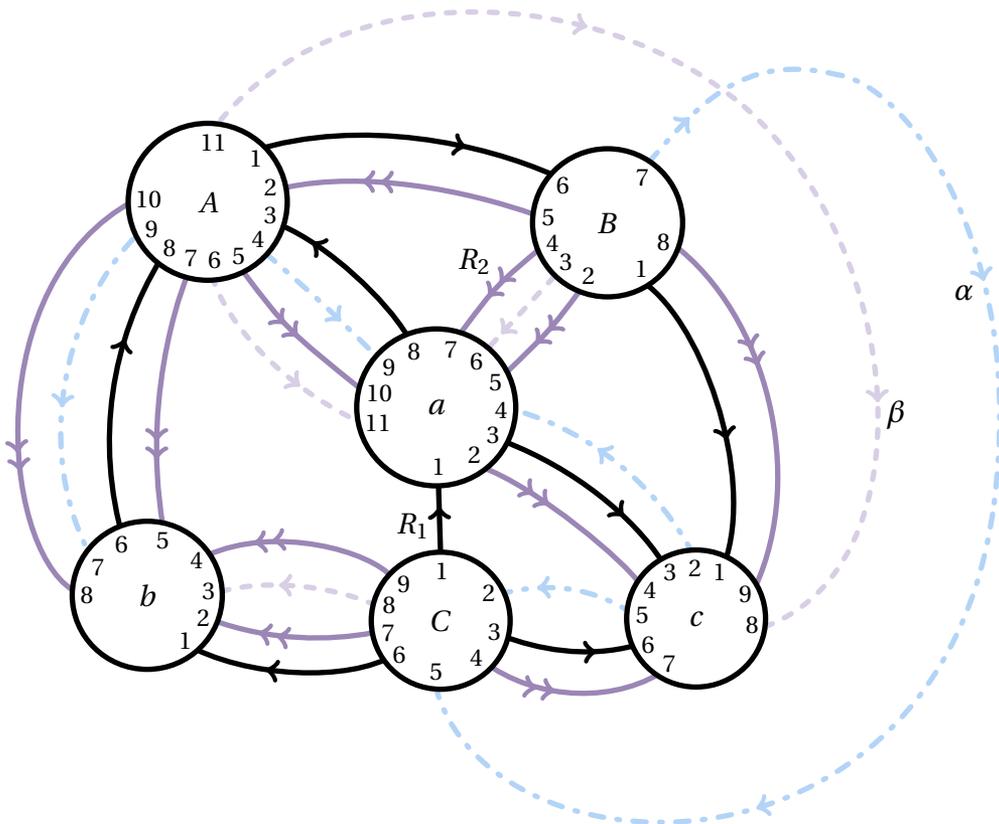

\begin{proof}[Proof of Theorem~\ref{thm:comp}]
  The 22 manifolds are specified by particular triangulations that are
  included in \cite{ancillary}. For each triangulation $\cT$, we
  proved the following:
  \begin{enumerate}
  \item \textsl{The manifold $M$ underlying $\cT$ is hyperbolic. }
  
    We used \cite{hikmot2013verified} to find intervals which are
    guaranteed to contain shapes for the tetrahedra of $\cT$ which
    give rise to an actual complete hyperbolic structure on $M$.
  
  \item \textsl{The triangulation $\cT$ is the canonical cellulation of $M$.}
  
    Doing arithmetic with the interval shapes as described in \S
    \ref{sec:certone}, we verified that all the inequalities in
    Theorem~\ref{thm:weeks} hold, and hence $\cT$ is canonical.

  \item \textsl{ $M$ is asymmetric.}
  
    We used SnapPy to find all combinatorial self-isomorphisms of
    $\cT$; as there was only the identity, asymmetry follows from
    Corollary~\ref{cor:asym}.

  \item \textsl{ $M$ has two lens space fillings of coprime order.}
  
    We used SnapPy to check that the $(1,0)$ and $(0,1)$ Dehn fillings
    on $M$ (with respect to the cusp framing specified in the
    triangulation file for $\cT$) have fundamental groups with
    presentations that are obviously those of finite cyclic groups of
    coprime order; the Geometrization Theorem implies that these are
    lens spaces.  This combinatorial step is also performed rigorously
    by SnapPy. (Regina \cite{Regina} can go further and identify the
    particular lens spaces directly, without appealing to
    geometrization; this data is included in
    Table~\ref{table:examples}.)
  \end{enumerate}
  To finish off Theorem~\ref{thm:comp}, it remained to show that
  the examples are distinct.  For this, we checked that no
  two of the triangulations were combinatorially isomorphic.  By
  (b), this implies the 22 manifolds are not isometric and hence not
  homeomorphic. Alternatively, this is proved in \cite{Burton2014}
  by different methods.

  Complete source code for this proof is available at
  \cite{ancillary}.  As a precaution, two disjoint subsets of the
  authors wrote independent implementations of step (b), and the
  entire proof was executed from a single script.  Additionally, our
  code is robust enough to run on all 59{,}107 one-cusped census
  manifolds in \cite{Burton2014}; excluding the 64 cases where SnapPy
  believes there are canonical cells which are not tetrahedra, we were
  able to certify the canonical triangulations for all of these manifolds.
\end{proof}

We next extend the phenomena exhibited in Theorem~\ref{thm:comp} to an
infinite family of examples; note that our conventions for Dehn
filling are specified in Figure~\ref{fig:link}.

\theoremlink

\begin{proof}
  Let $N_k$ denote the $(6k+1, k)$ Dehn filling
  on the second cusp of $N$; we focus on this case first for notational
  simplicity, leaving the $(6k-1, k)$ Dehn filling for later.
  The theorem in this case follows immediately from the next two
  lemmas.
  \begin{lemma}\label{lem:linkasym}
    For all $k \in \Z$ with $\abs{k}$ sufficiently large, the manifold
    $N_k$ is hyperbolic and asymmetric.
  \end{lemma}

  \begin{lemma}\label{lem:linkfilling}
    For all $k \in \Z$, the $(1,0)$ and $(4,1)$ Dehn fillings on the
    remaining cusp of $N_k$ are lens spaces of coprime orders $\abs{6k+1}$
    and $\abs{15k + 4}$, respectively.
  \end{lemma}

  \begin{proof}[Proof of Lemma \ref{lem:linkasym}]
    Let $\cT$ be the particular triangulation of
    $N$ included in \cite{ancillary}.   Using \cite{hikmot2013verified} and
    Proposition~\ref{prop:multicanon}, we verified that $\cT$ is in fact
    the canonical triangulation of $N$.  The triangulation $\cT$ has
    only two combinatorial isomorphisms: the identity and one that
    interchanges the two cusps.  Hence by Corollary~\ref{cor:asym}
    the only isometry of $N$ that preserves the each cusp is the
    identity.  The lemma now follows from the argument used to
    prove Lemma~\ref{lem:hyperdehn}.
  \end{proof}

  \begin{remark}\label{rem:referee}
    The referee kindly pointed out that the link $L$ in
    Figure~\ref{fig:link} is the Montesinos link
    $M(0; (5,3), (3, -2), (5,1)),$ and hence its symmetry group
    $\pi_0\big( \mathrm{Diff}(S^3, L)\big)$ can be computed using
    Boileau and Zimmermann \cite{BoileauZimmermann1987}. Unlike the
    case of knots \cite{GordonLuecke1989}, a symmetry of a link
    exterior need not send meridians to meridians; for example, the
    symmetry group of the $(-2, 3, 8)$-pretzel link is $\Z/2\Z$, but
    the symmetry group of its exterior has order 8.  While the proof
    of Lemma~\ref{lem:linkasym} given above requires that we know the
    full symmetry group of the exterior $N$, rather than just
    $\pi_0\big( \mathrm{Diff}(S^3, L)\big)$, by working harder one can
    prove Lemma \ref{lem:linkasym} from the results in
    \cite{BoileauZimmermann1987} without reference to a canonical
    triangulation of $N$.  We now sketch this alternative argument.

    Using \cite{BoileauZimmermann1987}, one computes that the symmetry
    group of the link $L$ is $\Z/2\Z$ where the generator interchanges
    the two components.  If infinitely many $N_k$ admit a nontrivial
    symmetry, then since the symmetry group of $N$ is finite, there is
    an infinite set of indices $k_i$ where said symmetry of $N_{k_i}$
    is induced by a fixed symmetry $f$ of $N$.  We will show that $f$
    is a symmetry of the underlying link $L$ and consequently $f$ must
    be the identity.  Let $C_1$ and $C_2$ be torus cross-sections for
    the two cusps of $N$.  For each $i$, the symmetry $f$ preserves
    the unoriented isotopy class of the Dehn filling curve
    $\gamma_i \subset C_2$ used to form $N_{k_i}$.  Again passing to a
    subsequence, we can assume that $f$ either preserves the
    \emph{oriented} isotopy class of all $\gamma_i$ or reverses the
    orientation on all of them.  In the former case, it follows that
    $f$ restricted to $C_2$ is isotopic to the identity; in the
    latter, it must be isotopic to the elliptic involution.  Since the
    two components of $L$ have nonzero linking number, the maps
    $H_1(C_i ; \Z) \to H_1(N; \Z)$ are both injective; it follows that
    the action of $f$ on $C_2$ determines the action of $f$ on
    $H_1(C_1; \Z)$.  Consequently, in either case, the map $f$ must
    preserve the unoriented isotopy classes of the meridians which
    record $L$ in both $C_1$ and $C_2$, and hence comes from a
    symmetry of $(S^3, L)$ as claimed.
  \end{remark}

  \begin{proof}[Proof of Lemma \ref{lem:linkfilling}]
    It is clear from Figure~\ref{fig:link} that both link components
    are unknotted; the $(1,0)$ filling on $N_k$ is thus a lens
    space of order $\abs{6k+1}$.

    Turning now to the other filling, let $P$ denote the $(4,1)$
    filling of the first cusp of $N$.  The key idea is that $P$
    is Seifert fibered over the disc with two exceptional fibers of
    orders 3 and 5, and hence has infinitely many lens space Dehn
    fillings; we chose the fillings defining the
    $N_k$ to be these lens space slopes.  If $\pair{\mu,
      \lambda}$ is a meridian-longitude basis (with respect to the
    standard link framing) for $\pi_1(\partial P)$, SnapPy easily
    computes that
    \[
    \pi_1(P) = \spandef{a, b}{b^3a^5 = 1} \mtext{with $\mu = b^2a^2$
      and $\beta \assign \mu^6 \lambda = b^3 = a^{-5}$.}
    \]
    In particular, the $(6k+1, k)$ filling, which is along the slope
    $\mu \beta^k$, has fundamental group $\spandef{a,b}{b^3a^5 = 1, \
      \mu \beta^k = b^2a^{2-5k} = 1}$.  Replacing the second relator
    by its product on the left with the inverse of the first relator
    yields the following presentation:
    \[
    \spandefm{\big}{a,b}{b^3 a^5 = 1, \ a^{-(5k+3)} = b} \ = \ 
      \spandefm{\big}{a}{a^{15k + 4} = 1}
    \]
    Thus the $(4, 1)$ filling on $N_k$ is a lens space whose first
    homology has order $\abs{15k+4}$.  We conclude the proof for the
    $(6k+1, k)$ filling by noting that $p_1 = 6k + 1$ and $p_2 = 15
    k + 4$ are coprime, since $-(5k+3)p_1 + (2k + 1)p_2 = 1 \text{ for any }k$.
   \end{proof}

 The $(6k-1, k)$ case differs only in that the lens
  spaces have order $p_1' = 6k-1$ and $\abs{p_2'}$, where
  $p_2' = 15k - 4$.  These are coprime since $-(5k-3)p_1'
  + (2k-1)p_2' = 1$.
\end{proof}

{\RaggedRight \bibliographystyle{nmd/math}
  \small
  \bibliography{NoSymLSpace} }
\end{document}